\documentclass{amsart}

\newcommand{\reg}{\mbox{reg}\,}

\newcommand{\pd}{\mbox{pd}\,}

\newcommand{\h}{\mbox{ht}\,}
\newcommand{\bight}{\mbox{bight}\,}

\newtheorem{theorem}{Theorem}[section]
\newtheorem{corollary}[theorem]{Corollary}
\newtheorem{lemma}[theorem]{Lemma}

\newtheorem{definition}[theorem]{Definition}

\newtheorem{example}[theorem]{Example}

\newtheorem{conjecture}[theorem]{Conjecture}

\numberwithin{equation}{section}

\begin{document}
\bibliographystyle{amsplain}

\title[Closed neighborhood ideal of a graph]{Closed neighborhood ideal of a graph}
\author[L. Sharifan and S. Moradi]{Leila Sharifan and Somayeh Moradi}

\address{Leila Sharifan, Department of Mathematics, Hakim Sabzevari University, P.O.Box 397, Sabzevar, Iran.}
\email{l.sharifan@hsu.ac.ir}
\address{Somayeh Moradi, Department of Mathematics, School of Science, Ilam University, P.O.Box 69315-516, Ilam, Iran.} \email{so.moradi@ilam.ac.ir}


\begin{abstract}
We introduce a family of squarefree monomial ideals associated to finite simple graphs, whose monomial generators correspond to closed neighborhood of vertices of the underlying graph. Any such ideal is called the closed neighborhood ideal of the graph. We study some algebraic invariants of these ideals like Castelnuovo-Mumford regularity and projective dimension and present some combinatorial descriptions for these invariants in terms of graph invariants.
\end{abstract}

\maketitle

\section{Introduction}

One approach to study algebraic properties of monomial ideals is via combinatorial algebraic techniques which associate some combinatorial objects to the ideal. Any squarefree monomial ideal $I$ can be considered as the edge ideal of some hypergraph whose edges correspond to minimal generators of $I$. Another way to know more about monomial ideals, is to associate to an object like a graph, a simplicial complex, a lattice, etc., a monomial ideal and to find the relation between algebraic properties of the ideal and the data from the object. As the first class of such kind of ideals, the edge ideals of graphs were introduced by Villarreal (see \cite{V}) and have been studied widely. Later, some other families of ideals of polynomial rings associated to graphs like path ideals, binomial edge ideals and $t$-clique ideals were introduced and studied in \cite{CD,HHK,Moradi}, respectively. Studying some homological invariants of  ideals  like Castelnuovo-Mumford regularity and projective dimension in terms of combinatorial invariants of the underlying combinatorial object has been of great interest and specially for the edge ideals of graphs some nice classes of graphs for which such descriptions exist, have been presented.

In this paper, we introduce the closed neighborhood ideal of a graph $G$ and try to calculate some of its algebraic invariants like regularity and projective dimension, in terms of information from $G$.
We fix some notation that we use in the paper. Throughout this paper, $G $ is a finite simple graph with the vertex set $V(G)$ and the edge set $E(G)$. If $V (G) = \{x_1, \ldots , x_n\}$, we identify the
vertices of the graph with the variables in the polynomial ring $R = {\bf{k}}[x_1, \ldots, x_n]$ where  $\bf{k}$
is a fixed field. The closed neighborhood ideal of $G$ denoted by $NI(G)$, is defined as an ideal in $R$ generated by all monomials of the form $\prod_{x_j\in N[x_i] } x_j$, where $x_i\in V(G)$. Here $N[x_i]$ is the closed neighborhood of $x_i$ in $G$, which is the set $\{x_i\}\cup \{x_j\ ; \ \{x_i,x_j\}\in E(G)\}$. By a \textit{path graph} $P_n$ we mean a graph on the vertex set $\{1,\ldots ,n\}$ with the edge set $E(G)=\{\{1,2\},\{2,3\},\ldots,\{n-1,n\}\}$.
Recall that a subset $S\subseteq V(G)$ is called a \textit{dominating set} of $G$, if $S\cap N_G[x]\neq \emptyset$  for any vertex $x$ of $G$. Also a dominating set $S$ of $G$ is called a \textit{minimal dominating set} if no proper subset of $S$ is a dominating set of $G$.  The \textit{domination number} of $G$, denoted $\gamma(G)$, is the minimum size of a dominating set of $G$. A subset $M$ of $E(G)$ is called a \textit{matching} for $G$, if any two edges in $M$ are disjoint and a matching $M$ is called a \textit{maximal matching} of $G$ if it is not properly contained in another matching of $G$. The \textit{matching number} of $G$ is the maximum size of a matching in $G$ and we denote it by $a_G$. For a subset $W\subseteq V(G)$, the induced subgraph of $G$ on the vertex set $W$ is denoted by $G|_W$. A subset $F\subseteq V(G)$ is called an \textit{independent set} of $G$, if no edge of $G$ is contained in $F$.  For a monomial ideal $I$ of $R$, the \textit{big height} of $I$ is defined as the maximum height of minimal prime ideals of $I$ and is denoted by $\bight(I)$.

The paper proceeds as follows.  In Theorem \ref{forest}, we show that when $G$ is a forest, the matching number $a_G$ is a lower bound for both $\reg(R/NI(G))$ and $\pd(R/NI(G))$. Then for special families of graphs, namely path graphs, generalized star graphs, $m$-book graphs and the family of graphs described in Theorem \ref{special} it is proved that $\reg(R/NI(G))=a_G$ (see Theorems \ref{path}, \ref{special}, \ref{star} and \ref{m-book}). Moreover, we study the projective dimension of $R/NI(G)$ for some of these classes. Finally in Theorem \ref{r-partite}, for any complete $r-$partite graph $G$, we obtain the regularity and projective dimension of $NI(G)$ and show that $R/NI(G)$ is sequentially Cohen-Macaulay.

\section{Main results}

The closed neighborhood ideal of a graph $G$ is defined as follows.

\begin{definition}
Let $G$ be a graph. The {\it closed neighborhood ideal} of $G$ is the monomial ideal
$$NI(G)=\langle \prod_{x_j\in N[x_i] } x_j\ : \ x_i\in V(G)\rangle.$$
\end{definition}

Set $\gamma'_G=max\{|C|:\ C\ \textrm{is a minimal dominating set of} \ G\}$.
One can see the following easy lemma.

\begin{lemma}\label{minprime}
Any minimal prime ideal of $NI(G)$ is of the form $\langle x_{j_1},\ldots,x_{j_r}\rangle$, where $\{x_{j_1},\ldots,x_{j_r}\}$ is a  minimal dominating set of $G$. Therefore
$$\h(NI(G))=\gamma(G), \ \bight(NI(G))=\gamma'_G$$ and $$\h(NI(G)^\vee)=\min\{\deg_G(x_i)\ : \ x_i\in V(G)\}+1.$$
\end{lemma}
\begin{proof}
Note that since $NI(G)$ is a monomial ideal, any minimal prime ideal of $NI(G)$ is generated by some variables. The result is obtained noticing the fact that for any $S=\{x_{j_1},\ldots,x_{j_r}\}\subseteq V(G)$, $NI(G)\subseteq \langle x_{j_1},\ldots,x_{j_r}\rangle$ if and only if $S$ is a dominating set of $G$.
\end{proof}

Inspired by Lemma \ref{minprime}, we define the {\it dominating ideal} of $G$ as the monomial ideal
$$DI(G)=\langle \prod_{x_j\in S} x_j\ : \ S\ {\text{is a minimal dominating set of}}\ G\rangle.$$

From Lemma \ref{minprime}, we deduce that $NI(G)^\vee=DI(G)$.

\begin{example}
Let $G$ be a graph on the vertex set $\{x_1,\ldots ,x_5\}$  and $$E(G)=\{\{x_1,x_2\},\{x_2,x_3\},\{x_3,x_4\},\{x_4,x_5\}\}$$ then
$$NI(G)=\langle x_1x_2,x_2x_3x_4,x_4x_5\rangle$$ and
$$DI(G)=\langle x_2x_4,x_2x_5,x_1x_4,x_1x_3x_5\rangle$$
\end{example}

The following is an easy consequence of \cite[Corollary 3.33]{MV} and Lemma \ref{minprime}.
\begin{corollary}
For any graph $G$, $\pd(R/NI(G))\geq \gamma'_G$.
\end{corollary}

In the sequel, we study the regularity and the projective dimension of $NI(G)$ for some families of graphs. The following theorem gives lower bounds for these invariants, for a forest $G$, in terms of the matching number of $G$.

\begin{theorem}\label{forest}
Let $G$ be a forest. Then $\reg(R/NI(G))\geq a_G$ and $\pd(R/NI(G))\geq a_G$.
\end{theorem}

\begin{proof}
We prove the assertion by induction on $|V(G)|$.
For any vertex $w\in V(G)$, set $u_w=\prod_{z\in N_G[w]} z$.
Let $x\in V(G)$ be a leaf in $G$ and $y\in N_G(x)$. Then $NI(G)=\langle  xy\rangle+ J$, where $J=\langle u_w:\ w\in V(G)\setminus \{x,y\}\rangle$. We have $J:y=J: xy=NI(G\setminus \{x,y\})$.
Hence by \cite[Corollary 2.6]{Sh} and considering the short exact sequence
$0\rightarrow R/(J:xy) (-2)\rightarrow R/J \rightarrow R/NI(G)\rightarrow 0$, we have
\begin{equation}\label{1}
 \reg( R/NI(G))=\max\{\reg( R/NI(G''))+1,\reg( R/J)\},
\end{equation}
and
\begin{equation}\label{1}
 \pd( R/NI(G))=\max\{\pd( R/NI(G''))+1,\pd( R/J)\},
\end{equation}
where $G''=G\setminus \{x,y\}$.
It is enough to show that $\reg( R/NI(G''))+1\geq a_G$ and $\pd( R/NI(G''))+1\geq a_G$.
By induction hypothesis $\reg( R/NI(G''))\geq a_{G''}$ and $\pd( R/NI(G''))\geq a_{G''}$.
Note that if we add the edge $e=\{x,y\}$ to any matching of $G''$, we get a matching of $G$ with one more edge. So $a_{G''}+1\leq a_G$.
Let $M=\{e_1,\ldots,e_{a_G}\}$ be a maximal matching of $G$ with $a_G=|M|$. If  $e\in M$, then $M\setminus \{e\}$ is a matching of $G''$ and then $a_{G''}\geq a_G-1$. If  $e\notin M$, then one must has $e'=\{y,z\}\in M$ for some $z\in N_G(y)$, $z\neq x$, since otherwise $M\cup \{e\}$ will be a matching of $G$, contradicting to the maximality of $M$. Now, $M'=(M\setminus \{ e'\})$ is a matching of $G''$ of size $a_G-1$, and then $a_{G''}\geq a_G-1$.
So $a_{G''}+1=a_G$.
Hence
\begin{equation}\label{2}
 \reg( R/NI(G''))+1\geq a_{G''}+1=a_G.
\end{equation}
and \begin{equation}\label{2}
 \pd( R/NI(G''))+1\geq a_{G''}+1=a_G.
\end{equation}
\end{proof}

Next we are going to compute homological invariants of  $NI(G)$ when $G$ is a path graph. Note that in this case $NI(G)$ is an ideal of decreasing type (\cite[Definition 2.7]{Sh}) and it is possible to compute the graded Betti numbers by iterated mapping cone technique. But since this ideal is closely related to the path ideal $I_3(P_n)$, by applying the results of \cite{AF}, we just compute the projective dimension and the regularity.

\begin{theorem}\label{path}
Let $P_n$ be a path graph. Then
\begin{itemize}
\item[(i)] $\h(NI(P_n))=[n+2/3]$ and $\h(DI(P_n))=3$.
\item[(ii)] $\reg(R/NI(P_n))=[n/2]$ and
\begin{eqnarray*}
\pd(R/NI(P_n)) = \left\{
\begin{array}{ll}
n/2,  &   \text{if $n$ is even};\\
(n+1)/2, &  \text{if $n$ is odd.}\\
\end{array}\right.
\end{eqnarray*}
\end{itemize}
\end{theorem}
\begin{proof}
(i) follows from Lemma \ref{minprime}.

(ii) In order to compute projective dimension and regularity, first note that by \cite[Corollary 4.15]{AF}, if $n=4p+d$, where $1\leq d\leq 3$, then
\begin{eqnarray*}
\pd(R/I_3(P_n)) = \left\{
\begin{array}{ll}
2p,  &   \text{$d\neq 3$};\\
2p+1, &  \text{$d=3$}\\
\end{array}\right.
\end{eqnarray*}
and
\begin{eqnarray*}
\reg(R/I_3(P_n)) = \left\{
\begin{array}{ll}
2p,  &   \text{$d\neq 3$};\\
2(p+1), &  \text{$d=3$.}\\
\end{array}\right.
\end{eqnarray*}

It is easy to see that $NI(P_n)=J+\langle x_1x_2,x_{n-1}x_n\rangle$, where
$$J=\langle x_2x_3x_4,x_3x_4x_5,\ldots , x_{n-3}x_{n-2}x_{n-1}\rangle\cong I_3(P_{n-2}).$$

Set $T_{n-1}=J+\langle x_1x_2\rangle$. Then by \cite[Corollary 2.6]{Sh}, the minimal free resolution of $R/T_{n-1}$ is obtained by applying the mapping cone technique to the short exact sequence
$$0\to R/(J: x_1x_2)(-2)\to R/J\to R/T_{n-1}\to 0.$$
Moreover, $J: x_1x_2\cong T_{n-3}$. So, $$\reg(R/T_{n-1})=\max\{\reg(R/I_3(P_{n-2})),\reg(R/T_{n-3})+1\}$$ and
$$\pd(R/T_{n-1})=\max\{\pd(R/I_3(P_{n-2})),\pd(R/T_{n-3})+1\}.$$
From the formula given for projective dimension and regularity of $R/I_3(P_n)$ and induction we conclude that $\pd(R/T_n)=\reg(R/T_n)=[n/2]$.

Since  $NI(P_n)=T_{n-1}+\langle x_{n-1}x_n\rangle$, again  by \cite[Corollary 2.6]{Sh}, the minimal free resolution of $R/NI(P_n)$ is given by the mapping cone technique for the short exact sequence
$$0\to R/(T_{n-1}: x_{n-1}x_n)(-2)\to R/T_{n-1}\to R/NI(P_n)\to 0.$$
Moreover, $T_{n-1}: x_{n-1}x_n\cong NI(P_{n-2}).$ So, $$\reg( R/NI(P_n))=\max\{\reg(R/T_{n-1}),\reg(R/NI(P_{n-2})+1\}$$ and
$$\pd(R/NI(P_n))=\max\{\pd(R/T_{n-1})),\pd(R/NI(P_{n-2})+1\}.$$
From the formula obtained for projective dimension and regularity of $R/T_n$, by induction we get the desired formula for $NI(P_n)$.
\end{proof}

Let $G$ be an arbitrary graph on the vertex set $V$. We set $$V_0(G)=\{x\in V:\ \deg_G(x)=0\},$$
$$V_1(G)=\{x\in V:\ \deg_G(x)=1\},$$ $$V_2(G)=\{x\in V:\ \deg_G(x)\geq 2,\ N_G(x)\cap V_1(G)\neq \emptyset\}$$ and
$$V_3(G)=\{x\in V:\ \deg_G(x)\geq 2,\ N_G(x)\cap V_1(G)= \emptyset\}.$$  It is clear that $V=V_0(G)\dot\cup V_1(G)\dot\cup V_2(G)\dot\cup V_3(G)$. In the following we are going to use this notations in order to study regularity and projective dimension of $NI(G)$ for a special class of graphs.

\begin{theorem}\label{special}
Let $G$ be a graph on the vertex set $V=\{x_1,\ldots,x_n\}$. Assume that $V_3(G)$ is an independent set of $G$. Then
\begin{itemize}
\item[(i)] $a_G=|V_2(G)|+\frac{1}{2}|V'_1(G)|$, where $V'_1(G)=\{x\in V_1(G):\ N_G(x)\subseteq V_1(G)\}$,
\item[(ii)] $\reg(R/NI(G))=a_G$,
\item[(iii)]  $\pd(R/NI(G)=\bight(NI(G))=n-a_G$.
\end{itemize}
\end{theorem}
\begin{proof}
(i) Note that for each $x\in V'_1(G)$ there exists a unique $y\in V'_1(G)$ such that $\{x,y\}\in E(G)$. In fact the induced subgraph of $G$ on the vertex set $V'_1(G)$ is a disjoint union of some edges, say $\{e_1,\ldots,e_{\frac{1}{2}|V'_1(G)|}\}$, that appear in each maximal matching of $G$.
For each $x\in V_2(G)$ choose $y_x\in V_1(G)$ such that $\{x,y_x\}\in E(G)$.  It is clear that if $x\neq x'\in V_2(G)$ then $y_x\neq y'_x$. This shows that
$$T=\{e_1,\ldots,e_{\frac{1}{2}|V'_1(G)|}\}\cup \{\{x,y_x\}:\ x\in V_2(G)\}$$ is a matching of $G$ of size $|V_2(G)|+\frac{1}{2}|V'_1(G)|$. Since $V_3(G)$ is an independent set of $G$, it is clear that
$T$ is a maximal matching of $G$.  We say that a maximal matching $S$ of $G$ is of type $(*)$ if $S$ is a disjoint union of $\{e_1, \ldots, e_{\frac{1}{2}|V'_1(G)|}\}$ and a set of edges $e=\{x,y\}$ whose endpoints belong to $V_2(G)$ and $V_1(G)$ ($x\in V_2(G)$ and $y\in V_1(G)$). If $S$ is of type $(*)$, it is clear that $|S|=|T|$.

Now let $S$ be an arbitrary maximal matching of $G$ with $|S|=a_G$. As discussed above, \ \ $\{e_1,\ldots,e_{\frac{1}{2}|V'_1(G)|}\}\subset S$. Assume that for some $x\in V_3(G)$ and $e\in S$, $x\in e$. Then there exists $x'\in V_2(G)$ such that $e=\{x,x'\}$. Since $N_G(x')\cap V_1(G)\neq \emptyset$ one can choose $y_{x'}\in V_1(G)$ such that $\{x',y_{x'}\}\in E(G)$. One can easily check that $S_1=(S\setminus \{e\})\cup\{\{x',y_{x'}\}\}$ is a maximal matching of $G$ of size $a_G$. Continuing in this way, we finally get a maximal matching of size $a_G$ which is of type $(*)$. So $a_G=|V_2(G)|+\frac{1}{2}|V'_1(G)|$ and the conclusion follows.

(ii) We prove the assertion by induction on the number of elements of $V_3(G)$.
First assume that $|V_3(G)|=0$. Let $V_2(G)=\{x_{i_1},\ldots,x_{i_s}\}$ and without loss of generality, assume that
$$N_G(x_{i_1})\cap V_1(G)=\{y_{1},\ldots,y_{t_1}\},$$
$$N_G(x_{i_2})\cap V_1(G)=\{y_{t_1+1},\ldots, y_{t_2}\},$$
$$.$$ $$.$$ $$.$$
$$N_G(x_{i_s})\cap V_1(G)=\{y_{t_{s-1}+1},\ldots, y_{t_s}\}.$$

One can easily check that $$NI(G)=I_1+\cdots+I_s+NI(G|_{V'_1(G)})+ \langle x_i:\ x_i\in V_0(G)\rangle$$ where $I_j=\langle x_{i_j}y_k:\ t_{j-1}+1\leq k\leq t_j\rangle$ (Here $t_0=0$).
So
\begin{multline*}
R/NI(G)\cong R_1\otimes \cdots \otimes R_s\otimes {\bf{k}}[x_i:\ x_i\in V'_1(G)]/NI(G|_{V'_1(G)})\\ \otimes
 {\bf{k}}[x_i:\ x_i\in V_0(G)]/\langle x_i:\ x_i\in V_0(G)\rangle,
\end{multline*}
where $R_j={\bf{k}}[x_{i_j},y_k:\ t_{j-1}+1\leq k\leq t_j]/I_j$.
Now since $\reg(R_j)=1$ for any $1\leq j\leq s$, $\reg( {\bf{k}}[x_i:\ x_i\in V'_1(G)]/NI(G|_{V'_1(G)}))=\frac{1}{2}|V'_1(G)|$ and
$\reg({\bf{k}}[x_i:\ x_i\in V_0(G)]/\langle x_i:\ x_i\in V_0(G)\rangle)=0$, we have $$\reg(R/NI(G))=\sum_{j=1}^s \reg(R_j)+\frac{1}{2}|V'_1(G)|=|V_2(G)|+\frac{1}{2}|V'_1(G)|=a_G.$$

Now let $|V_3(G)|=m\geq 1$ and assume inductively that the assertion is true for any graph $H$ that $|V_3(H)|<m$ and $V_3(H)$ is an independent set of $H$.
Choose $x\in V_3(G)$. Then
\begin{equation}\label{eq4}
NI(G)=NI(G\setminus\{x\})+\langle f\rangle,
\end{equation}
where $f=\prod_{x_j\in N[x] } x_j$ is a monomial of degree $\deg_G(x)+1$. Note that
\begin{equation}\label{eq1}
 NI(G\setminus\{x\}):f=NI(G\setminus N[x]).
\end{equation}

Both $G\setminus\{x\}$ and $G\setminus N[x]$ satisfy the induction hypothesis, $|V_3(G\setminus\{x\})|<m$ and $|V_3(G\setminus N[x])|<m$. So by induction hypothesis,
$\reg(R/NI(G\setminus\{x\}))=a_{G\setminus\{x\}}$ and $\reg(R/NI(G\setminus N[x]))=a_{G\setminus N[x]}$.
Moreover, by the formula that we have already proved for the matching number,

\begin{equation}\label{eq2}
a_{G\setminus\{x\}}=|V_2(G\setminus\{x\})|+\frac{1}{2}|V'_1(G\setminus\{x\})|=|V_2(G)|+\frac{1}{2}|V'_1(G)|=a_G
\end{equation}
and
\begin{multline}\label{eq3}
a_{G\setminus N[x]}=|V_2(G\setminus N[x])|+\frac{1}{2}|V'_1(G\setminus N[x])|\\ =|V_2(G)|+\frac{1}{2}|V'_1(G)|)-\deg_G(x)=a_G-\deg_G(x).
\end{multline}

Now consider the short exact sequence
\begin{equation}\label{eq5}
0\to R/(NI(G\setminus\{x\}):f)(-\deg(f))\to R/NI(G\setminus\{x\})\to R/NI(G)\to 0.
\end{equation}
Since $x\in supp(f)$ and for any minimal monomial generator $g$ of $NI(G\setminus\{x\})$, $x\not\in supp(g)$, by \cite[Corollary 2.6]{Sh} and equations (\ref{eq1}), (\ref{eq2}) and (\ref{eq3}), we have
\begin{multline*}
\reg(R/NI(G))=\\\max\{\reg(R/(NI(G\setminus\{x\}):f))+\deg(f)-1, \reg(R/NI(G\setminus\{x\}))\}=a_G.
\end{multline*}

\medskip

(iii) First we show that $\bight(NI(G))=n-a_G$.
By Lemma \ref{minprime} and (i) it is enough to prove that
$$\gamma'(G)= |V_0(G)|+|V_3(G)|+|V_1(G)-V'_1(G)|+\frac{1}{2}|V'_1(G)|.$$

First note that if $S$ is an arbitrary minimal dominating set of $G$, then $V_0(G)\subset S$ and for each $e\in E(G|_{V'_1(G)})$, $|e\cap S|=1$.
Actually, for each maximal independent set $W$ of $G|_{V'_1(G)}$, we have $|W|=\frac{1}{2}|V'_1(G)|$ and $S_W=W\cup (V_1(G)-V'_1(G))\cup V_0(G)\cup V_3(G)$ is a minimal dominating set of $G$ which is of size $|V_0(G)|+|V_3(G)|+|V_1(G)\setminus V'_1(G)|+\frac{1}{2}|V'_1(G)|=n-a_G$.
Hence $\gamma'(G)\geq n-a_G$.

Now let $T$ be an arbitrary minimal dominating set of $G$. By the discussion of the previous paragraph, $W=T\cap V'_1(G)$ is a maximal independent set of $G|_{V'_1(G)}$ and $V_0(G)\subset T$. Note that if $T\cap V_2(G)=\emptyset$ then $T=S_W$ and so $|T|=n-a_G$. Suppose that  $T\cap V_2(G)\neq \emptyset$. So, for each $x\in T\cap V_2(G)$ and each $y\in N_G(x)\cap V_1(G)$, $y\not\in T$. Thus for each $x\in T\cap V_2(G)$ we can choose  $y_x\in (N_G(x)\cap V_1(G))\setminus T$. Consider the map $\phi: T\to S_W$ with
\begin{eqnarray*}
\phi(x) = \left\{
\begin{array}{ll}
y_x,  &   \text{if $x\in T\cap V_2(G)$};\\
x, &  \text{if $x\in T\setminus V_2(G).$}\\
\end{array}\right.
\end{eqnarray*}

It is clear that $\phi$ is one-to-one and so $|T|\leq S_W=n-a_G$. This implies that $\gamma'(G)=n-a_G$.

Now we are ready to prove that $\pd(R/NI(G)=n-a_G$. We prove the assertion by induction on the number of elements of $V_3(G)$.
First assume that $|V_3(G)|=0$. Using the notation in the proof of (ii), we
have
\begin{multline*}
R/NI(G)\cong R_1\otimes \cdots \otimes R_s\otimes {\bf{k}}[x_i:\ x_i\in V'_1(G)]/NI(G|_{V'_1(G)})\\ \otimes
 {\bf{k}}[x_i:\ x_i\in V_0(G)]/\langle x_i:\ x_i\in V_0(G)\rangle,
\end{multline*}

where $R_j={\bf{k}}[x_{i_j},y_k:\ t_{j-1}+1\leq k\leq t_j]/I_j$ and $\pd(R_j)=t_j=|N_G(x_{i_j})\cap V_1(G)|$.
Now since $\pd( {\bf{k}}[x_i:\ x_i\in V'_1(G)]/NI(G|_{V'_1(G)}))=\frac{1}{2}|V'_1(G)|$ and $\pd({\bf{k}}[x_i:\ x_i\in V_0(G)]/\langle x_i:\ x_i\in V_0(G)\rangle)=|V_0(G)|$, the conclusion follows for the case that $|V_3(G)|=0$ and in this case we have $\pd(R/I)=\sum_{j=1}^s \pd(R_j)+\frac{1}{2}|V'_1(G)|+|V_0(G)|=|V_1(G)\setminus V'_1(G)|+\frac{1}{2}|V'_1(G)|+|V_0(G)|=n-a_G$.

Now let $|V_3(G)|=m\geq 1$. Assume that by induction hypothesis the assertion is true for each graph $H$ that $|V_3(H)|<m$ and $V_3(H)$ is an independent set of $H$.
Choose $x\in V_3(G)$.  By (\ref{eq4}) and (\ref{eq1}),  $NI(G)=NI(G\setminus\{x\})+\langle f\rangle$, where $f=\prod_{x_j\in N[x] } x_j$ and  $NI(G\setminus\{x\}):f=NI(G\setminus N[x])$.
By induction hypothesis,
$\pd(R/NI(G\setminus\{x\}))=n-1-a_{G\setminus\{x\}}=n-1-a_{G}$ and $\pd(R/NI(G\setminus N[x]))=n-\deg_G(x)-1-a_{G\setminus N[x]}=n-1-a_G$. Again, from the short exact sequence (\ref{eq5}) we have
$$\pd(R/NI(G))=\max\{\pd(R/(NI(G\setminus\{x\}):f))+1, \pd(R/NI(G\setminus\{x\}))\}=n-a_G.$$
\end{proof}

For integers $n_1,\ldots,n_k>1$, let $S_{n_1,\ldots,n_k}$ be a graph obtained by gluing $k$ path graphs $P_{n_1},\ldots,P_{n_k}$ at one end point of each path. We call $S_{n_1,\ldots,n_k}$ a \textit{generalized star graph}.

\begin{theorem}\label{star}
Let $G=S_{n_1,\ldots,n_k}$ for positive integers $n_1,\ldots,n_k$. Then $$\reg(R/NI(G))=a_G.$$
\end{theorem}

\begin{proof}
Since $G$ is a tree, by  Theorem \ref{forest}, it is enough to show that $\reg(R/NI(G))\leq a_G$.  We prove this inequality by induction on $|V(G)|$. For $|V(G)|=2$ the result is clear. Suppose that $|V(G)|>2$ and the result holds for any generalized star graph with less than $|V(G)|$ vertices. Let $P_{n_1}: x_1,x_2,\ldots,x_{n_1}$ be such that the gluing point of $P_{n_1}$ in $G$ is $x=x_{n_1}$. Set $I=NI(G))$.  We have $I=\langle  x_1x_2\rangle+ J$, where $J=\langle u_w:\ w\in V(G)\setminus \{x_1,x_2\}\rangle$ and $u_w=\prod_{z\in N_G[w]} z$.
Also $J:x_2=J: x_1x_2=NI(G\setminus \{x_1,x_2\})$.
Considering the short exact sequence
$0\rightarrow R/(J:x_1x_2) (-2)\rightarrow R/J \rightarrow R/I\rightarrow 0$, we have
\begin{equation}\label{5}
 \reg( R/I)\leq \max\{\reg( R/NI(G_2))+1,\reg( R/J)\},
\end{equation}
where $G_2=G\setminus \{x_1,x_2\}$.
By induction hypothesis $\reg( R/NI(G_2))\leq a_{G_2}$.  As was shown in the proof of Theorem \ref{forest}, $a_{G_2}+1=a_G$. Thus it is enough to show that
$\reg( R/J)\leq a_G$. Considering the short exact sequence
$0\rightarrow R/(J:x_2) (-1)\rightarrow R/J \rightarrow R/\langle J,x_2\rangle\rightarrow 0$, we have
\[
 \reg( R/J)\leq \max\{\reg( R/NI(G_2))+1,\reg( R/\langle J,x_2\rangle)\}\leq \max\{a_G,\reg( R/\langle J,x_2\rangle)\}.
\]
Set $J_1=J$ and $J_r=\langle u_w:\ w\in V(G)\setminus \{x_i: 1\leq i\leq r+1\}\rangle$ for any $2\leq r\leq n_1-1$. Then  $\langle J,x_2\rangle =\langle J_2,x_2\rangle$ and $\reg( R/\langle J,x_2\rangle)=\reg( R/J_2)$. Therefore
\begin{equation}\label{6}
\reg( R/J)\leq \max\{a_G,\reg( R/J_2)\}.
\end{equation}

Similar to what we have done for $J$ and considering the short exact sequence
$0\rightarrow R/(J_2:x_3) (-1)\rightarrow R/J_2 \rightarrow R/\langle J_2,x_3\rangle\rightarrow 0$, we have
\begin{equation}\label{7}
 \reg( R/J_2)\leq \max\{\reg( R/(J_2:x_3))+1,\reg( R/\langle J_2,x_3\rangle)\}\leq \max\{a_G,\reg( R/\langle J_2,x_3\rangle)\}.
\end{equation}
Note that $(J_2:x_3)=NI(G_3)$, where $G_3=G\setminus \{x_i:\ 1\leq i\leq 3\}$ and then by induction hypothesis, $\reg( R/(J_2:x_3))\leq a_{G_3}< a_G$.

Also $\langle J_2,x_3\rangle=\langle J_3,x_3\rangle$, where $J_3=\langle u_w:\ w\in V(G)\setminus \{x_i: 1\leq i\leq 4\}\rangle$ and then $\reg( R/\langle J_2,x_3\rangle)=\reg( R/J_3)$. Thus by (\ref{6}) and (\ref{7}),   $ \reg( R/J)\leq  \max\{a_G,\reg( R/J_3)\}$. Proceeding this way we get
\[
\reg( R/J)\leq  \max\{a_G,\reg( R/J_{n_1-1})\}.
\]
Note that $J_{n_1-1}=\langle u_w:\ w\in V(G)\setminus V(P_{n_1})\rangle$.
So
\begin{equation}\label{8}
 \reg( R/J_{n_1-1})\leq \max\{\reg( R/(J_{n_1-1}:x))+1,\reg( R/\langle J_{n_1-1},x\rangle)\}.
\end{equation}
Also $(J_{n_1-1}:x)=\sum_{i=2}^k NI(P_{n_i-1})$ and then
\[
\reg( R/(J_{n_1-1}:x))+1=\sum_{i=2}^k \reg( R/NI(P_{n_i-1}))+1=\sum_{i=2}^k a_{P_{n_i-1}}+1\leq a_G.
\]
The first equality comes from the fact that the ideals $NI(P_{n_i-1})$
for $i=2,\ldots,k$ live in polynomial rings with pairwise disjoint variables. The second equality holds by Theorem \ref{path}.

Moreover, $\langle J_{n_1-1},x\rangle=\langle \sum_{i=2}^k T_{n_i-1},x\rangle$, where $T_{n_i}$ is as defined in the proof of Theorem \ref{path}. As was shown in that proof, we have $\reg( R/T_{n_i-1})=[(n_i-1)/2]= a_{P_{n_i-1}}$ for any $i$. Thus
$\reg(R/\langle J_{n_1-1},x\rangle)=\sum_{i=2}^k\reg( R/T_{n_i-1})= \sum_{i=2}^k a_{P_{n_i-1}}<a_G$. Therefore using (\ref{8}), $ \reg( R/J_{n_1-1})\leq a_G$ and then $\reg(R/J)\leq a_G$. The proof is complete.

\end{proof}

The \textit{$m$-book graph} is defined as the graph Cartesian product of $S_{m+1}$ and $P_2$, where $S_m$ is a star graph on $m$ vertices.
We denoted the $m$-book graph by $B_m$.
\begin{theorem}\label{m-book}
Let $G=B_m$ for some positive integer $m$. Then $$\reg(R/NI(G))=a_G=m+1.$$
\end{theorem}

\begin{proof}
We may assume that $V(G)=\{x_0,x_1,\ldots,x_m,y_0,y_1,\ldots,y_m\}$ and $E(G)=\{\{x_0,x_i\}:\  1\leq i\leq m\}\cup \{\{y_0,y_i\}:\  1\leq i\leq m\}\cup \{\{x_i,y_i\}:\  0\leq i\leq m\}$. Thus $NI(G)=\langle f_i,g_i:\  1\leq i\leq m+1\rangle$, where $f_i=x_0x_iy_i$ and $g_i=y_0x_iy_i$ for any $1\leq i\leq m$, $f_{m+1}=y_0\prod_{i=0}^m x_i$ and $g_{m+1}=x_0\prod_{i=0}^m y_i$.
For any $1\leq i\leq m$, set $J_i=\langle f_j,g_j:\  1\leq j\leq i\rangle$. Then $NI(G)=J_m+\langle f_{m+1},g_{m+1} \rangle$.

By induction on $i$, we prove that $\reg(R/J_i)=i+1.$ One can easily see that $\reg(R/J_1)=2$. Now, let $\reg(R/J_i)=i+1$. We prove that $\reg(R/J_{i+1})=i+2$. Note that $J_{i+1}=J_i+\langle f_{i+1},g_{i+1} \rangle$. Also $J_i:f_{i+1}=\langle x_1y_1,x_2y_2,\ldots,x_iy_i\rangle $ and $(J_i+\langle f_{i+1}\rangle):g_{i+1}=\langle x_1y_1,x_2y_2,\ldots,x_iy_i,x_0\rangle$.
Consider the short exact sequence
\[
0\rightarrow R/(J_i:f_{i+1}) (-3)\rightarrow R/J_i \rightarrow R/(J_i+\langle f_{i+1} \rangle)\rightarrow 0.
\]
Since $J_i:f_{i+1}$ is generated by a regular sequence, we have $\reg( R/(J_i:f_{i+1}))=i$. Thus by \cite[Corollary 2.6]{Sh}, we have $\reg(R/(J_i+\langle f_{i+1} \rangle))=\max\{\reg(R/J_i),\reg(R/(J_i:f_{i+1}))+2\}=i+2$.
Consider the short exact sequence
\[
0\rightarrow R/((J_i+\langle f_{i+1}\rangle):g_{i+1}) (-3)\rightarrow R/(J_i+\langle f_{i+1}\rangle) \rightarrow R/J_{i+1}\rightarrow 0.
\]
Since  $(J_i+\langle f_{i+1}\rangle):g_{i+1}$ is generated by a regular sequence, we have $\reg(R/((J_i+\langle f_{i+1}\rangle):g_{i+1}))=i$.
Again by \cite[Corollary 2.6]{Sh}, we have $\reg(R/J_{i+1})=\max\{\reg(R/(J_i+\langle f_{i+1}\rangle)),\reg(R/(J_i+\langle f_{i+1}\rangle):g_{i+1})+2\}=i+2$.
Therefore $\reg(R/J_m)=m+1$. One can easily see that $J_m:f_{m+1}=\langle y_1,\ldots,y_m\rangle$ and $(J_m+\langle f_{m+1}\rangle):g_{m+1}=\langle x_1,\ldots,x_m\rangle$. Now, applying regularity formulas to the short exact sequences
\[
0\rightarrow R/(J_m:f_{m+1}) (-m-2)\rightarrow R/J_m \rightarrow R/(J_m+\langle f_{m+1} \rangle)\rightarrow 0
\]
and
\[
0\rightarrow R/((J_m+\langle f_{m+1}\rangle):g_{m+1}) (-m-2)\rightarrow R/(J_m+\langle f_{m+1}\rangle) \rightarrow R/NI(G)\rightarrow 0
\]

we get $\reg(R/NI(G))=a_G=m+1.$
\end{proof}

Finally, we study the dominating ideal and the closed neighborhood ideal of complete $r-$partite graphs.

\begin{theorem}\label{r-partite}
Let $G$ be a complete $r-$partite graph on the vertex set $$V(G)=\{x_{1,1},\ldots ,x_{1,n_1}\}\dot\cup \cdots \dot\cup \{x_{r,1},\ldots ,x_{r,n_r}\},$$ where $n_1\leq \cdots\leq n_r$. Then
\begin{itemize}
\item[(i)]
$$DI(G)=I(G)+\langle x_{j,1}\cdots x_{j,n_j}\ : \ 1\leq j\leq r\rangle,$$
where $I(G)$ is the edge ideal of $G$.
\item[(ii)]
$\h(DI(G))=n-n_r+1$ and $\h(NI(G))=\min\{n_1,2\}$.
\item[(iii)]
If for some $1\leq s\leq r$, $n_1=\cdots=n_s=1$ and $n_{s+1}>1$, then
$$DI(G)=DI(G\setminus \{x_{1,1},x_{2,1},\ldots,x_{s,1}\})+\langle x_{1,1},x_{2,1},\ldots,x_{s,1}\rangle.$$
\item[(iv)] If $n_1>1$, then
$\beta_{i,j}(R/DI(G))=\beta_{i,j}(R/I(G))+\gamma_{i,j}$, where \begin{align*}&\beta_{i,j}(R/I(G))= \nonumber\\
&
\left\{
  \begin{array}{ll}
   \sum_{\ell=2}^{r}(\ell-1)\sum_{\alpha_1+\cdots+\alpha_\ell=i+1, \ j_1<\cdots <j_\ell,\alpha_1,\ldots, \alpha_\ell\geq 1}{n_{j_1}\choose \alpha_1}\cdots{n_{j_\ell}\choose \alpha_\ell}, & \hbox{if $j=i+1$} \\
   0, & \hbox{if $j\neq i+1$}
  \end{array}
\right.
\end{align*}

and
$$\gamma_{i,j}=\sum_{{n_t=j-i+1,t=1}}^r{n-n_t\choose i-1}.$$
\item[(v)]
If for some $1\leq s\leq r$, $n_1=\cdots=n_s=1$ and $n_{s+1}>1$ then
\begin{multline*}
\beta_{i,j}(R/DI(G))=\\
\sum_{\ell=0}^i {s\choose i-\ell} \beta_{\ell , j-i+\ell}\big({\bf{k}}[V(G\setminus\{x_{1,1},x_{2,1},\ldots,x_{s,1}\})]/
DI(G\setminus \{x_{1,1},x_{2,1},\ldots,x_{s,1}\})\big).
\end{multline*}
\item[(vi)]
$DI(G)$ is a componentwise linear ideal.
\item[(vii)]
$R/NI(G)$ is sequentially Cohen-Macaulay.
\item[(viii)]
$\reg(DI(G))=\pd(R/NI(G))=n_r$.
\item[(ix)]
$\pd(R/DI(G))=\reg(NI(G))=\left\{
   \begin{array}{ll}
     n-1, &    \hbox{if\ $\max\{i\ : n_i=1\}\leq r-2$;} \\
     r, & \hbox{otherwise.}
   \end{array}
 \right.$

\end{itemize}
\end{theorem}
\begin{proof}
The parts (i) and (iii) are straightforward by definition of dominating ideal and (ii) follows by Lemma \ref{minprime}.

(iv)  If $n_1>1$, then the union of the minimal set of monomial generators of $I(G)$ and $\{x_{j,1}\cdots x_{j,n_j}\ : \ 1\leq j\leq r\}$ is a minimal system of generators for $DI(G)$. Moreover, the Betti numbers of $I(G)$ are given in  \cite[Theorem 5.3.8]{J} as we have presented  in (iv). Looking at these Betti numbers, we conclude that $I(G)$ has a linear resolution. So, by \cite[Theorem 3.2]{HHZh}, $I(G)$ has linear quotients and it is a componentwise linear ideal. Also, for each $1\leq i\leq r$, one can easily check that
\begin{multline*}
L_i=(I(G)+ \langle x_{j,1}\cdots x_{j,n_j}\ : \ 1\leq j\leq i-1\rangle): \langle x_{i,1}\cdots x_{i,n_i}\rangle\\
=\langle x_{t,s}\ :\ 1\leq t\leq r, t\neq i, 1\leq s\leq n_t\rangle.
\end{multline*}
So  $DI(G)$ has linear quotients and the conclusion follows by \cite[Theorem 2.6]{ShV}.

(v) One can easily see that
\begin{multline*}
R/DI(G)\cong {\bf{k}}[x_{1,1},\ldots ,x_{s,1}]/\langle x_{1,1},\ldots ,x_{s,1}\rangle\\
\otimes {\bf{k}}[V(G\setminus\{x_{1,1},x_{2,1},\ldots,x_{s,1}\})]/
DI(G\setminus \{x_{1,1},x_{2,1},\ldots,x_{s,1}\}).
\end{multline*}

(vi) By (iii) and the proof of (iv),  $DI(G)$ has linear quotients. So. by \cite[Corollary 2.4]{ShV}, it is a componentwise linear ideal.

(vi) This part follows by  (vi) and \cite[Theorem 8.2.20]{HH}.

(viii) and (ix) follow by (iv), (v) and \cite[Theorem 2.1]{T}.
\end{proof}

We end the paper with the following conjecture.

\begin{conjecture}
Let $G$ be a forest. Then $\reg(R/NI(G))=a_G$.
\end{conjecture}

{}

\begin{thebibliography}{}
\bibitem{AF} A. Alilooee and S. Faridi, {Graded Betti numbers of path ideals of cycles and lines}, Journal of Algebra and Its Applications, \textbf{ 17}    (2018), 1850011.

\bibitem{CD}  A. Conca, E. De Negri, { M-Sequences, graph ideals and ladder ideals of linear type,} J. Algebra \textbf{211}
(1999), no. 2, 599--624.

\bibitem{HHK} J. Herzog, T. Hibi, F. Hreinsdottir, T. Kahle, J. Rauh, {Binomial edge ideals and
conditional independence statements}, Advances in Applied Mathematics, \textbf{45} (2010)
317-–333.

\bibitem{HH} J. Herzog, T. Hibi, Monomial Ideals. Graduate Texts in Mathematics. Springer, Berlin (2011).

\bibitem{HHZh} J. Herzog, T. Hibi and X. Zheng,
\newblock Monomial ideals whose powers have
a linear resolution,
\newblock { Math. Scand.} \textbf{95} (2004), no. 1, 23--32.
\bibitem{J} S, Jacques,
 \newblock Betti Numbers of Graph Ideals, PhD thesis,
 \newblock 	arXiv:math/0410107, 2004.


\bibitem{Moradi}  S. Moradi, {$t$-clique ideal and $t$-independence ideal of a graph,} Communications in Algebra \textbf{46} ( 2018), no. 8, 3377--3387.


 \bibitem{MV}  S. Morey, R. H. Villarreal, { Edge ideals: Algebraic and combinatorial properties,}
Progress in commutative algebra, Combinatorics and homology \textbf{1}, (2012), 85--126.


\bibitem{Sh} L. Sharifan, Minimal Free Resolution of Monomial Ideals by Iterated
Mapping Cone, Bulletin of the Iranian Mathematical Society  \textbf{44} (2018), Issue 4,  1007–-1024.

\bibitem{ShV} L. Sharifan and M. Varbaro, Graded Betti numbers of ideals with linear quotients, Le Mathematiche LXIII (2008), 257--265.

\bibitem{T}N. Terai, Alexander duality theorem and Stanley–Reisner rings. Srikaisekikenkysho Kkyroku \textbf{107}
 (1999), 174–-184.
\bibitem{V} R. H. Villarreal, { Cohen-Macaulay graphs.} Manuscripta Math. {\bf 66} (1990), no. 3, 277--293.

\end{thebibliography}
\end{document}